\documentclass[12pt]{article} 

\usepackage{amsmath}
\usepackage{amsthm}
\usepackage{amsfonts}
\usepackage{mathrsfs}
\usepackage{stmaryrd}
\usepackage{setspace}
\usepackage{fullpage}
\usepackage{amssymb}
\usepackage{breqn}
\usepackage{enumitem}
\usepackage{bbold} 
\usepackage{authblk}
\usepackage{comment}
\usepackage{hyperref}
\usepackage{pgf,tikz}
\usepackage{graphicx}
\usepackage{subcaption}

\newtheorem*{thm*}{Theorem}
\newtheorem{thm}{Theorem}[section]

\newtheorem{lemma}[thm]{Lemma}

\newtheorem*{prop*}{Proposition}
\newtheorem{prop}[thm]{Proposition}

\newcommand\ex{\ensuremath{\mathrm{ex}}}

\newcommand\cK{{\mathcal K}}

\newcommand\cN{{\mathcal N}}

\newcommand{\ignore}[1]{}

\title{Some stability and exact results in generalized Turán problems}
\author{Dániel Gerbner\\\small Alfr\'ed R\'enyi Institute of Mathematics\\
\small \texttt{gerbner.daniel@renyi.hu}}

\date{}

\begin{document}

\maketitle

\begin{abstract}
Given graphs $H$ and $F$, the generalized Tur\'an number $\mathrm{ex}(n,H,F)$ is the largest number of copies of $H$ in $n$-vertex $F$-free graphs. Stability refers to the usual phenomenon that if an $n$-vertex $F$-free graph $G$ contains almost $\mathrm{ex}(n,H,F)$ copies of $H$, than $G$ is in some sense similar to some extremal graph. We obtain new stability results for generalized Turán problems and derive several new exact results. 
    
\end{abstract}




\section{Introduction}

A fundamental question in graph theory is the following. Given a graph $F$, what is the largest number of edges that an $n$-vertex $F$-free graph can have? This quantity is called the \textit{Tur\'an number} and is denoted by $\ex(n,F)$. Tur\'an \cite{T} proved that $\ex(n,K_{k+1})=|E(T(n,k))|$, where the \textit{Tur\'an graph} $T(n,k)$ is the complete $k$-partite graph with each part of order $\lfloor n/k\rfloor$ or $\lceil n/k\rceil$. Let us call a vertex or an edge of a graph \textit{color-critical} if deleting that vertex or edge decreases the chromatic number. Simonovits \cite{sim} showed that if $\chi(F)=k+1$ and $F$ has a color-critical edge, then for sufficiently large $n$, we have $\ex(n,F)=|E(T(n,k))|$.

Given graphs $H$ and $G$, we denote by $\cN(H,G)$ the number of copies of $H$ in $G$. The \textit{generalized Tur\'an number} $\ex(n,H,F)$ is the largest $\cN(H,G)$ in $n$-vertex $F$-free graphs $G$. The first such result is due to Zykov \cite{zykov}, who showed that $\ex(n,K_r,K_{k+1})=\cN(K_r,T(n,k))$. 

Gerbner and Palmer \cite{gp2} defined \textit{$F$-Tur\'an-good graphs} as graphs $H$ with $\ex(n,H,F)=\cN(H,T(n,\chi(F)-1))$ for sufficiently large $n$. In this language Zykov's theorem implies that cliques are $K_k$-Tur\'an-good (note that his result holds for every $n$). We say that $H$ is \textit{weakly $F$-Tur\'an-good} if $\ex(n,H,F)=\cN(H,T)$ for some complete $(\chi(F)-1)$-partite graph $T$. In this case a straightforward optimization finds $T$ for given $H$, but we are unable to execute this optimization for general $H$. Gy\H ori,  Pach and Simonovits \cite{gypl} showed that complete multipartite graphs are weakly $K_k$-Tur\'an-good. They also showed a bipartite graph which is not $K_3$-Tur\'an-good. 

Gerbner and Palmer \cite{gp2} showed that $C_4$ is $F_2$-Tur\'an-good, where $F_2$ consists of two triangles sharing a vertex. Note that $F_2$ does not have a color-critical edge, thus the Tur\'an-graph is not edge-maximal: one more edge can be added without creating an $F_2$, but that edge is not in any $C_4$. Gerbner \cite{gerbner2} showed that there are $F$-Tur\'an-good graphs if and only if $F$ has a color-critical vertex.

Let us call $H$ \textit{asymptotically $F$-Tur\'an-good} if $\ex(n,H,F)=(1+o(1))\cN(H,T(n,\chi(F)-1))$, and \textit{weakly asymptotically $F$-Tur\'an-good} if $\ex(n,H,F)=(1+o(1))\cN(H,T)$ for some complete $(\chi(F)-1)$-partite graph $T$. A theorem of Gerbner and Palmer \cite{gp2} states that if $\chi(F)=k$, then $\ex(n,H,F)\le \ex(n,H,K_k)+o(n^{|V(H)|}$. It
implies that if $\chi(H)\le k$ and $H$ is (weakly) asymptotically $K_{k+1}$-Tur\'an-good, then $H$ is also (weakly) asymptotically $F$-Tur\'an-good. Another theorem of Gerbner and Palmer \cite{gp3} states that paths are asymptotically $F$-Tur\'an-good for non-bipartite graphs $F$.

\smallskip


In extremal graph theory problems often the graphs where a parameter takes its maximum all have a similar structure. A very common phenomenon is that graphs where the parameter is close to its maximum are in some sense close to the extremal graphs. The prime example is the following theorem.

\begin{thm}[Erd\H os-Simonovits stability theorem \cite{erd1,erd2,simi}]\label{sta}
Let $\chi(F)=k+1$. If $G$ is an $n$-vertex $F$-free graph with $|E(G)|\ge \ex(n,F)-o(n^{2})$, then $G$ can be obtained from $T(n,k)$ by adding and removing $o(n^2)$ edges.
\end{thm}


\smallskip

Let us turn to known stability results concerning generalized Tur\'an problems. As most of these results are similar to Theorem \ref{sta}, we introduce a notation.

Let $\chi(H)<\chi(F)=k+1$.
We say that $H$ is \textit{$F$-Tur\'an-stable} if the following holds. If $G$ is an $n$-vertex $F$-free graph with $\cN(H,G)\ge \ex(n,H,F)-o(n^{|V(H)|})$, then $G$ can be obtained from $T(n,k)$ by adding and removing $o(n^2)$ edges. Theorem \ref{sta} states that $K_2$ is $F$-Tur\'an-stable for every non-bipartite $F$. We say that $H$ is \textit{weakly $F$-Tur\'an-stable} if the following holds. If $G$ is an $n$-vertex $F$-free graph with $\cN(H,G)\ge \ex(n,H,F)-o(n^{|V(H)|})$, then $G$ can be obtained from a complete $k$-partite graph by adding and removing $o(n^2)$ edges. We remark that if $H$ is (weakly) $F$-Tur\'an-stable, then $H$ is (weakly) asymptotically $F$-Tur\'an-good.

Ma and Qiu \cite{mq} obtained the following generalization of Theorem \ref{sta}.

\begin{thm}[Ma, Qiu \cite{mq}]\label{maqi}
Let $r<\chi(F)$. Then $K_r$ is $F$-Tur\'an-stable.
\end{thm}

A different kind of stability result is due to the author \cite{ge}. Assume that $F$ has a color-critical edge and $\chi(F)>r$. If $G$ is an $n$-vertex $F$-free graph with chromatic number more than $\chi(F)-1$, then $\ex(n,K_r,F)-\cN(K_r,G)=\Omega(n^{r-1})$.


Most other stability results in this area were obtained either as a lucky coincidence when the proof of a bound on $\ex(n,H,F)$ gives a stronger result (e.g. $\ex(n,P_4,C_5)$ in \cite{gp3}), or as a lemma towards a bound on $\ex(n,H,F)$ (e.g. $\ex(n,C_5,K_k)$ in \cite{lm}). 

A more systematic approach to this latter type of stability results is due to Hei, Hou and Liu \cite{hhl}. They showed that if $\chi(H)<\chi(F)$, $F$ has a color-critical edge, the Tur\'an graph contains the most copies of $H$ among complete $k$-partite graphs and $H$ is $F$-Tur\'an-stable, then $H$ is $F$-Tur\'an-good. In fact, instead of $F$-Tur\'an-stability, they used the following weaker property: If $G$ is an $n$-vertex $F$-free graph with $\cN(H,G)= \ex(n,H,F)$, then $G$ can be obtained from $T(n,k)$ by adding and removing $o(n^2)$ edges\footnote{Their statement is weaker. They use is later for $H$ satisfying additional properties, and they assume those properties in this theorem, but those properties are not actually used in the proof.}. They also showed that if $F$ has a color-critical edge, then paths are $F$-Tur\'an-stable\footnote{Their statement is again weaker. They use it later in the above described way, thus they only need to consider $n$-vertex $F$-free graphs $G$ with $\cN(P_k,G)=\ex(n,P_k,F)$, but their proof also works for $G$ with $\cN(P_k,G)\ge \ex(n,P_k,F)-o(n^k)$.}.

Liu, Pikhurko, Sharifzadeh and Staden \cite{lpss} introduced a general framework for studying graphs parameters that do not decrease by Zykov symmetrization, and proved that under some conditions, a stability result is also implied. Symmetrization does not decrease $\ex(n,H,K_k)$ if $H$ is complete multipartite. It is not hard to see that the additional conditions are also satisfied, thus their results imply that complete multipartite graphs are weakly $K_k$-Tur\'an-stable.

\smallskip

Now we are ready to list our contributions. 


\begin{prop}\label{kovv}
Let $\chi(F)=k+1$ and $H$ be weakly asymptotically $F$-Tur\'an-good and weakly $K_{k+1}$-Tur\'an-stable. Then $H$ is weakly $F$-Tur\'an-stable. Furthermore, if $H$ is asymptotically $F$-Tur\'an-good and $K_{k+1}$-Tur\'an-stable, then $H$ is $F$-Tur\'an-stable.
\end{prop}

Combining the above proposition with the known results mentioned earlier, we obtain that complete multipartite graphs and paths are $F$-Tur\'an-stable for every $F$ with larger chromatic number.

We can extend the result of Hei, Hou and Liu \cite{hhl} to the weak case.

\begin{thm}\label{haszn}
Let $\chi(F)>\chi(H)$ and assume that $F$ has a color-critical edge. If $H$ is weakly $F$-Tur\'an-stable, then $H$ is weakly $F$-Tur\'an-good. Moreover, the same porperty is also implies by the weaker assumption that there is an $n$-vertex $F$-free graph $G$ with $\cN(H,G)= \ex(n,H,F)$ can be obtained from a complete $(\chi(F)-1)$-partite graph by adding and removing $o(n^2)$ edges 
\end{thm}

Combined with results mentioned earlier, the above theorem implies that complete multipartite graphs $H$ are weakly $F$-Tur\'an-good if $\chi(H)<\chi(F)$ and $F$ has a color-critical edge. This was proved by the author \cite{ger} in the case $\chi(F)=3$. Note that if $H$ is $F$-Tur\'an-stable, then $H$ is weakly $F$-Tur\'an-good by the above theorem, but not necessarily $F$-Tur\'an-good: it is possible that the extremal graph is only slightly unbalanced. This is the case for $H=K_{1,3}$ and $F=K_3$; it was shown by Brown and Sidorenko \cite{brosid} that the bipartite graph with the most copies of $K_{1,3}$ is either $K_{k,n-k}$ or $K_{k+1,n-k-1}$, where $k=\lfloor \frac{n}{2}-\sqrt{(3n-4)/2}\rfloor$.

Theorem \ref{haszn} is the special case $r=1$ of the next theorem.

\begin{thm}\label{haszn2}
Let $\chi(F)=k+1$, $\chi(H)=k$ and assume that $F$ has a color-critical vertex. Let $r$ be the smallest number such that there is a color-critical vertex $v$ in $F$ that is adjacent to exactly $r$ vertices of one of the color classes in the $k$-coloring of the graph we obtain by deleting $v$ from $F$. Assume that if we embed graphs of maximum degree less than $r$ into each part of a complete $k$-partite graph $T_0$, we do not obtain any copies of $H$ besides those contained in $T_0$.

Assume that there is an $n$-vertex $F$-free graph $G$ with $\cN(H,G)= \ex(n,H,F)$, such that $G$ can be obtained from a complete $k$-partite graph $T$ by adding and removing $o(n^2)$ edges. Then $H$ is weakly $F$-Tur\'an-good.
\end{thm}

We remark that the assumptions of the theorem ensure that we can embed a graph with maximum degree $r-1$ into one part of $T$, but the resulting graph contains $\cN(H,T)$ copies of $H$.


\smallskip

Using stability, we can obtain a result on the structure of the extremal graphs for $\ex(n,H,F)$ in some cases.
Let $d_G(H,v)$ denote the number of copies of $H$ containing $v$ in $G$. Observe that for vertices $u,v$ in $T(n,k)$, we have $d_{T(n,k)}(H,v)=(1+o(1))d_{T(n,k)}(H,u)$.  

\begin{thm}\label{struc}
Let $\chi(H)<\chi(F)=k+1$, and assume that $H$ is $F$-Tur\'an-stable. 
Let $G$ be an $n$-vertex $F$-free graph which contains $\ex(n,H,F)$ copies of $F$. Then for every vertex $u\in G$ we have $d_G(H,u)\ge (1+o(1))d_{T(n,k)}(H,v)$.
\end{thm}

Note that for the ordinary Tur\'an case, this is known \cite{sim}. 
Let us now state some of the exact results of generalized Tur\'an problems that follow from our theorems (together with known results).

\begin{thm}\label{resu}

\textbf{(i)} If $\chi(F)=k+1$, $\chi(F)$ has a color-critical edge, $H$ is 
$K_{k+1}$-Tur\'an-stable, then $H$ is weakly $F$-Tur\'an-good. In particular, complete $r$-partite graphs and paths are weakly $F$-Tur\'an-good. If $H$ is also $K_{k+1}$-Tur\'an-good, then $H$ is also $F$-Tur\'an-good.

\textbf{(ii)} If $\chi(F)=k+1$, $\chi(F)$ has a color-critical vertex and $H$ is a complete $k$-partite graph with each part sufficiently large, then $H$ is weakly $F$-Tur\'an-good.

\textbf{(iii)} Let $F$ have chromatic number at least 4 and a color-critical edge. Then $C_5$ is $F$-Tur\'an-good.

\textbf{(iv)}  If $F$ has a color-critical edge, then any union of cliques of order less than $\chi(F)$ is $F$-Tur\'an-good. 

\textbf{(v)} Let us assume that $F$ has a color-critical edge and $\chi(F)=k+1$. The $K_{k+1}$-Tur\'an-good graphs listed in Theorem \ref{turgoo} are also $F$-Tur\'an-good.
\end{thm}

We remark that \textbf{(iv)} generalizes a theorem of the author \cite{gerb}, who showed that matchings are $F$-Tur\'an-good if $F$ has a color-critical edge.

In the next section, we state the lemmas we need. We also state and prove a proposition that give stability for several generalized Tur\'an problem. We present the proofs of Proposition \ref{kovv} and of Theorems \ref{haszn2}, \ref{struc} and \ref{resu} in Section 3.

\section{Preliminaries}\label{2}

We will use the removal lemma \cite{efr}.

\begin{lemma}[Removal lemma]
If an $n$-vertex graph $G$ contains $o(n^{|V(H)|}$ copies of $H$, then we can delete $o(n^2)$ edges from $G$ to obtain an $H$-free graph.
\end{lemma}

We will combine this with a result of Alon and Shikhelman \cite{ALS2016}.

\begin{prop}\label{ash}
$\ex(n,K_k,F)=\Theta(n^k)$ if and only if $\chi(F)>k$.
\end{prop}

We will also use the following lemma.

\begin{lemma}\label{neww}
Let us assume that $\chi(H)<\chi(F)$ and $H$ is weakly $F$-Turán-stable, i.e. $\ex(n,H,F)=\cN(H,T)-o(n^{|V(H)|})$ for some complete multipartite $n$-vertex graph $T$. Then every part of $T$ has order $\Omega(n)$.
\end{lemma}

\begin{proof}
Assume the statement does not hold and let $A_1,\dots,A_{\chi(F)-1}$ be the parts of $T$. Let $A_1$ be the smallest part and $A_2$ be the largest part. Then $|A_2|\ge n/(\chi(F)-1)$. Let $U$ be a subset of $A_2$ of order $n/2(\chi(F)-1)$ and move $U$ from $A_2$ to $A_1$ to obtain another $F$-free graph $T'$. This way we remove only those copies of $H$ that contain an edge between $U$ and $A_2$, thus $o(n^{|V(H)|})$ copies. Let us take a proper $\chi(H)$-coloring of $H$, and embed a color class into $U$, another color class into $A_2\setminus U$, and the other color classes to $A_3,\dots, A_{\chi(H)}$. This way we obtain $\Theta(n^{|V(H)|})$ copies of $H$ that are in $T'$ and not in $T$, thus $\cN(H,T')>\cN(H,T)+\Theta(n^{|V(H)|}>\ex(n,H,F)$, a contradiction completing the proof.
\end{proof}

There are several results \cite{gypl,gp2,gerbner2} that take $F$-Tur\'an-good graphs as building blocks and form new $F$-Tur\'an-good graphs.

Let $\chi(H),\chi(H')\le k$. Let $H_1$ be the vertex-disjoint union of $H$ and $H'$. Assume first that $H$ contains a copy $X$ of $K_{k}$,
then let $H_2$ be the graph obtained by taking $H_1$ and a clique $Y$ in $H'$, and connecting vertices of $Y$ and $X$ arbitrarily. 
Assume now that $H$ has a unique $k$-coloring.
Let $H_3$ obtained by taking $H$ and a $K_k$ with with vertices $v_1,\dots,v_k$, and adding additional edges such that for every $i\le k$, there is a copy of $K_k$ in $H_3$ containing $v_i$, but not containing any $v_j$ for $j>i$ and assume that $\chi(H_3)=k$.

\begin{thm}\label{turgoo}
\textbf{(i)} (Gerbner, Palmer \cite{gp3}) Let $H'$ be a $K_{k+1}$-Tur\'an-good graph and $H=K_k$.
	Then $H_2$ is $K_{k+1}$-Tur\'an-good.
	
	\textbf{(ii)} (Gerbner \cite{gerbner2})  Let $H$ be a $K_{k+1}$-Tur\'an-good graph. Then $H_3$ is $K_{k+1}$-Tur\'an-good.
\end{thm}

Here we show how the same arguments as used in \cite{gp3} and \cite{gerbner2} give similar statements for (weakly) asymptotically Tur\'an-good and (weakly) Tur\'an-stable graphs. The first statement of the next proposition is complicated because we need to emphasize that the nearly extremal complete multipartite graph is the same graph for both $H$ and $H'$.

\begin{prop}\label{propi}
Let $\chi(H),\chi(H')\le k$ and $\chi(F)>k$. Let $\ex(n,H,F)=(1+o(1))\cN(H,T)$ and $\ex(n,H',F)=(1+o(1))\cN(H',T)$ for some complete $k$-partite $n$-vertex graph $T$. Then $\ex(n,H_1,F)=(1+o(1))\cN(H_1,T)$.
Let $\ex(n,H,K_{k+1})=(1+o(1))\cN(H,T)$ and $\ex(n,H',K_{k+1})=(1+o(1))\cN(H',T)$ for some complete $k$-partite $n$-vertex graph $T$. Then $\ex(n,H_2,K_{k+1})=(1+o(1))\cN(H_2,T)$.
If $H$ is asymptotically $K_{k+1}$-Tur\'an-good, then $H_3$ is also  asymptotically $K_{k+1}$-Tur\'an-good. Moreover, $H_3$ is $K_{k+1}$-Tur\'an-stable.

Furthermore, if $H$ or $H'$ are weakly $F$-Tur\'an-stable, then $H_1$ is also weakly $F$-Tur\'an-stable. Similarly, if $H$ or $H'$ are $K_{k+1}$-Tur\'an-stable, then $H_2$ is also weakly $K_{k+1}$-Tur\'an-stable.
\end{prop}

We remark that the stability of $H_i$ is implied if we have stability for any of $H$ and $H'$.  In the case of $H_3$, $H'$ is replaced by $K_k$, for which we have stability by Theorem \ref{maqi}, thus we have stability automatically for $H_3$.

\begin{proof}
Let $G$ be an $n$-vertex $K_{k+1}$-free graph.
We count the copies of $H_i$ by picking $H$, a vertex-disjoint $H'$, and then the additional edges. Clearly picking $H$ and then picking $H'$ can be done $(1+o(1))\cN(H,T(n,k))$ and $(1+o(1))\cN(H',T(n,k))$ ways, thus it is asymptotically maximized by the Tur\'an graph. This completes the proof for $H_1$. 

For $H_2$, we consider the bipartite graph $G'$ with $X$ and $Y$ embedded into $G'$ as the two parts, with the edges of $G$ between them. It was shown in \cite{gp2} that a matching covering $Y$ is missing from $G'$. On the other hand, in the Tur\'an graph only such a matching is missing, thus the Tur\'an graph also maximizes $G'$, meaning that the bipartite graph obtained this way in $G$ is a subgraph of the bipartite graph obtained this way in $T(n,k)$. This implies that the number of ways to pick the additional edges is maximized in $T(n,k)$ also.

For $H_3$, we pick a copy $K$ of $K_k$ and $H$. Then we need to finish the embedding of $H_3$ into $G$ by adding the additional edges. We claim that there is at most one way to do that. Indeed, we go through the vertices $v_i$ in $H$ in increasing order. When we pick $v_i$, there is a copy of $K_{k}$ in $H_3$ containing $v_i$ such that the other vertices are already embedded into $G$. Since $G$ is $K_{k+1}$-free, we have that those vertices have at most one common neighbor in $G$, thus there is at most one way to choose $v_i$. On the other hand, in the Tur\'an graph there is always a way to finish the embbeding (see the proof of Proposition 1.3 in \cite{gerbner2}), thus the number of ways to pick the additional edges is maximized in $T(n,k)$.

Finally, if $G$ cannot be obtained from $T$ (or $T(n,k)$) by adding and removing $o(n^2)$ edges, then $\cN(H,G)<\alpha\cN(H,T)$ (or $\cN(H,G)<\alpha\cN(H,T(n,k))$) for some $\alpha<1$ (or the same holds with $H'$ in place of $H$). thus the above calculations give the bounds $\cN(H_1,G)<(1+o(1))\alpha\cN(H_1,T)$ and $\cN(H_2,G)<(1+o(1))\alpha\cN(H_2,T(n,k))$, a contradiction. 
\end{proof}

We remark that we can have $H=H'$. In that case we obtain that two (and by induction any number of) copies of $H$ have the same asymptotically extremal complete multipartite graph $T$, and if we have stability for $H$, then we have it for the multiple copies. 

\section{Proofs}

Let us start with the proof of Proposition \ref{kovv} that we restate here for convenience.

\begin{prop*}
Let $\chi(F)=k+1$ and $H$ be weakly asymptotically $F$-Tur\'an-good and weakly $K_{k+1}$-Tur\'an-stable. Then $H$ is weakly $F$-Tur\'an-stable. Furthermore, if $H$ is asymptotically $F$-Tur\'an-good and $K_{k+1}$-Tur\'an-stable, then $H$ is $F$-Tur\'an-stable.
\end{prop*}



\begin{proof} Let $G$ be an $n$-vertex $F$-free graph.
We start by applying the removal lemma and Proposition \ref{ash} to obtain a $K_{k+1}$-free graph $G_0$ by removing $o(n^2)$ edges. As we removed $o(n^{|V(H)|})$ copies of $H$ this way, we have that $\cN(H,G_0)\ge \ex(n,H,F)-o(n^{|V(H)|})=(1+o(1))\cN(H,T)$ for some complete $k$-partite graph $T$. As $H$ is weakly $K_{k+1}$-Tur\'an-stable, this means that $G_0$ can be obtained from $T$ by adding and removing $o(n^2)$ edges, thus so does $G$. The furthermore part follows the same way, $T=T(n,k)$ in that case.
\end{proof}

We continue with the proof of Theorem \ref{haszn2} that we restate here for convenience.

\begin{thm*}
Let $\chi(F)=k+1$, $\chi(H)=k$ and assume that $F$ has a color-critical vertex. Let $r$ be the smallest number such that there is a color-critical vertex $v$ in $F$ that is adjacent to exactly $r$ vertices of one of the color classes in the $k$-coloring of the graph we obtain by deleting $v$ from $F$. Assume that if we embed graphs of maximum degree less than $r$ into each part of a complete $k$-partite graph $T_0$, we do not obtain any copies of $H$ besides those contained in $T_0$.

Assume that there is an $n$-vertex $F$-free graph $G$ with $\cN(H,G)= \ex(n,H,F)$, such that $G$ can be obtained from a complete $k$-partite graph $T$ by adding and removing $o(n^2)$ edges. Then $H$ is weakly $F$-Tur\'an-good.
\end{thm*}

\begin{proof}
Let $k=\chi(F)-1$
and $G$ be an $n$-vertex $F$-free graph with $\cN(H,G)= \ex(n,H,F)$, that can be obtained from a complete $k$-partite graph $T$ with parts $V_1, \dots, V_k$ by adding and removing $o(n^2)$ edges. We pick $T$ such that we need to add and remove the least number of edges this way. In particular, every vertex $v\in V_i$ is connected at least as many vertices in every $V_j$ a in $V_i$ (otherwise we could move $v$ to $V_j$).

Let $E$ denote the set of edges in $G$ that are not in $T$ (i.e., those inside a part $V_i$).
Let $r(u)$ denote the number of edges incident to $u$ in $T$ that are not in $G$, i.e. the missing edges between $u$ and vertices in other part. Then we have $\sum_{u\in V(G)} r(u)=o(n^2)$, thus there are $o(n)$ vertices $u$ with $r(u)=\Omega(n)$. Let $A$ denote the set of vertices with $r(u)=o(n)$ and $A_i=A\cap V_i$, then $|A_i|=|V_i|-o(n)$. By Lemma \ref{neww}, we have that $|A_i|=\Omega(n)$. For $u\in V_i\setminus A_i$, we have that $u$ is adjacent to $\Omega(n)$ vertices in every $V_j$, thus in every $A_j$.

Recall that there are edges $vv_1,\dots,vv_r$ in $F$ such that by deleting these edges we obtain a $k$-partite graph with $v,v_1,\dots,v_r$ in the same part. Let $f_1$ denote the order of that part and $f_2,\dots, f_k$ denote the order of the other parts. 

We claim that every vertex in $V_i$ is adjacent to less than $r$ vertices of $A_i$. Assume otherwise, without loss of generality let $uu_1,\dots,uu_r$ be edges with $u\in V_1$, $u_1,\dots,u_r\in A_1$.
Let $B_i$ denote the neighborhood of $u$ in $A_i$, then $|B_i|=\Omega(n)$ and every vertex of $A\setminus A_i$ is connected to $|B_i|-o(n)$ vertices of $B_i$. We pick $f_1-r-1$ other vertices in $A_1$. 
These $f_1$ vertices have $|B_2|-o(n)$ common neighbors in $B_2$, we pick $f_2$ of them, and so on. For every $i$, we pick $f_i$ vertices from $B_i$ that are joined to every vertex picked earlier. This is doable, since all but $o(n)$ vertices of $B_i$ are connected to each of the vertices picked earlier. This way we find a copy of $F$ in $G$, a contradiction.

Let $X$ be a smallest set of vertices inside $V(G)\setminus A$ such that every $K_{1,r}$ inside $E$ contains at least one vertex of $X$. By the above, 
$\sum_{u\in X} r(u)=\Omega(n|X|)$. On the other hand, there are at most $\binom{|X|}{2}$ edges of $E$ inside $X$, and at most $|X|(r-1)$ edges of $E$ go out from $X$. Since $|X|=o(n)$, we have that either $\sum_{u\in X} r(u)=o(n|X|)$ (a contradiction), or $|X|=0$. In the latter case we have that by adding the edges of $E$ to $T$ we do not obtain any new copies of $H$, thus $\cN(H,G)\le \cN(H,T)$, completing the proof.
\end{proof}

We remark that this an example for a proof giving a stability result as well: if $|X|>0$ then $G$ contains at most $\ex(n,H,F)-\Omega(n^{|V(H)|-1}|X|)$ copies of $H$. In particular in the setting of Theorem \ref{haszn}, if an $F$-free $n$-vertex graph $G$ has chromatic number more than $k$, then $G$ contains $\ex(n,H,F)-\Omega(n^{|V(H)|-1})$ copies of $H$.

Let us continue with Theorem \ref{struc} that we restate here for convenience.

\begin{thm*}
Let $\chi(H)<\chi(F)=k+1$, and assume that $H$ is $F$-Tur\'an-stable. 
Let $G$ be an $n$-vertex $F$-free graph which contains $\ex(n,H,F)$ copies of $F$. Then for every vertex $u\in G$ we have $d_G(H,u)\ge (1+o(1))d_{T(n,k)}(H,v)$.
\end{thm*}

\begin{proof}
$G$ can be transformed to $T(n,k)$ by adding and removing $o(n^2)$ edges. Let $V_1$ be one of the partite sets of the resulting Tur\'an graph. Let $f(v)$ denote the number of copies of $H$ that are removed this way and contain $v$. Then we have $\sum _{v\in V(G)}f(v)=o(n^{|V(H)|})$. Consider a set $S$ of $|V(F)|$ vertices in $V_1$ such that $\sum _{v\in S}f(v)$ is minimal. Then by averaging $\sum _{v\in S}f(v)\le \frac{|S|}{|V_1}\sum _{v\in V_1}f(v)=o(n^{r-1})$.

Now we apply a variant of Zykov's symmetrization \cite{zykov}. Let us consider copies of $H$ that contain exactly one vertex $s$ of $S$, and if $sv$ is an edge of the copy of $H$, then $v$ is in the common neighborhood of $S$ in $G$ (i.e., we do not use the edges from $s$ to vertices not in the common neighborhood of $S$). Let
$d_G(H,S)$ denote the number of such copies. Observe that each vertex of $S$ is in $\frac{d_G(H,S)}{|S|}$ such copies of $H$.

Let $x$ denote the number of copies of $H$ that contain $u$ and a vertex from $S$, then $x=O(n^{r-2})$. If $d_G(H,u)<\frac{d_G(H,S)}{|S|}-x$, then we remove the edges incident to $u$ from $G$ and connect $u$ to the common neighborhood of $S$. This way we do not create any copy of $F$, as the copy should contain $u$, but $u$ could be replaced by any vertex of $S$ that is not already in the copy to create a copy of $F$ in $G$. We removed $d_G(H,u)$ copies of $H$, but added at least $\frac{d_G(H,S)}{|S|}-x$ copies, a contradiction. 

Therefore, we have \[d_G(H,u)\ge d_G(H,S)-x\ge d_{T(n,k)}(H,S)-\sum_{v\in S}f(r,v)-x=d_{T(n,k)}(H,S)-o(n^{r-1}).\]

Since $S\subset V_1$, the common neighborhood of $S$ in $T(n,k)$ is the same as the neighborhood of any vertex of $S$, thus $d_{T(n,k)}(H,S)=d_{T(n,k)}(H,s)$, completing the proof.
\end{proof}

We continue with the proof of Theorem \ref{resu}, which is too long to restate here.

\begin{proof}[Proof of Theorem \ref{resu}]
The first sentence of \textbf{(i)} follows from combining Proposition \ref{kovv} and Theorem \ref{haszn}. In the case $H$ is also $K_{k+1}$-Tur\'an-good, we have that the Tur\'an graph maximizes the number of copies of $H$ among complete $k$-partite $n$-vertex graphs. If $H$ is complete multipartite, we apply the result of Liu, Pikhurko, Sharifzadeh and Staden \cite{lpss} mentioned in the introduction, stating that $H$ is $K_{k+1}$-Tur\'an-stable. If $H$ is a path, we apply the result of Hei, Hou and Liu \cite{hhl}.

To prove \textbf{(ii)}, we combine Proposition \ref{kovv} and Theorem \ref{haszn2} with the result of Liu, Pikhurko, Sharifzadeh and Staden \cite{lpss}. We need to show that the assumption on $H$ of Theorem \ref{haszn2} is satisfied. Let us assume that $G$ is built from complete $k$-partite graph $T_0$ with embedding a graph with maximum degree less than $r$. Let $U_1,\dots,U_k$ be the parts of $T_0$ and $V_1,\dots,V_k$ be the color classes of $H$. Let $u,v\in U_1$ and assume that $uv$ is an edge in a copy of $H$. Without loss of generality, $u\in V_1$, $v\in V_2$. Then at most $2r-2$ other vertices of $U_1$ can be in $H$. This means that every $U_i$ contains either at most one color class of $H$ or at most $2r$ vertices of $H$. If each color class of $H$ has order more than $2r$, this is impossible.

To prove \textbf{(iii)}, we use results of
Lidick{\`y} and Murphy \cite{lm}. They proved that for $k\ge 3$, $C_5$ is $K_{k+1}$-Tur\'an-good and $K_{k+1}$-Tur\'an-stable. This implies that $C_5$ is weakly $F$-Tur\'an-stable by Proposition \ref{kovv}, hence $C_5$ is weakly $F$-Tur\'an-good by Theorem \ref{haszn}, where $\chi(F)=k+1$. A weakly $F$-Tur\'an-good and $K_{k+1}$-Tur\'an-good graph is clearly $F$-Tur\'an-good, since the complete $k$-partite graph with the most copies of $C_5$ is the Tur\'an graph.

To prove \textbf{(iv)}, we apply induction on the number of components. Let us assume that the statement holds for graphs with at most $\ell$ components, let $H$ be the vertex-disjoint union of $\ell$ cliques and $H'$ be another clique. Then we can use Proposition \ref{propi} to show that the vertex-disjoint union $H_1$ of $H$ and $H'$ is $F$-Tur\'an-stable. Then Theorem \ref{haszn} implies that $H_1$ is weakly $F$-Tur\'an-good. Let us assume that the extremal complete multipartite graph $T$ contains parts $A$ and $B$ with $|A|<|B|-1$. Then we move a vertex from $B$ to $A$. 

We claim that the number of copies of $H_1$ does not decrease this way. Indeed, every copy of $H_1$ intersects $A\cup B$ in a matching and some isolated vertices. Matchings are $K_3$-Tur\'an-good (first shown in \cite{gypl}), thus their number does not decrease this way. Clearly, such intersections are extended the same number of times to a copy of $H_1$ with vertices from the other parts, hence the number of copies of $H_1$ also does not decrease.
Repeating this we eventually arrive to the Tur\'an graph without decreasing the number of copies of $H_1$, which completes the proof.

To prove \textbf{(v)}, recall that in Section 2, we described how to obtain graphs $H_2$ and $H_3$ starting from $H$ and $H'$. These are generalizations of the constructions in \cite{gp3,gerb}. Therefore, applying Proposition \ref{propi} we obtain that the graphs listed are $K_{k+1}$-Tur\'an-stable. Theorem \ref{haszn} imply that they are weakly $F$-Tur\'an-good.
A weakly $F$-Tur\'an-good and $K_{k+1}$-Tur\'an-good graph is clearly $F$-Tur\'an-good, completing the proof.
\end{proof}

\bigskip

\textbf{Funding}: Research supported by the National Research, Development and Innovation Office - NKFIH under the grants KH 130371, SNN 129364, FK 132060, and KKP-133819.


\begin{thebibliography}{99}

		
		\bibitem{ALS2016} N. Alon, C. Shikhelman. Many $T$ copies in $H$-free graphs. \textit{Journal of Combinatorial Theory, Series B}, \textbf{121}, 146--172, 2016.
		

\bibitem{brosid} J. I. Brown, A. Sidorenko. The inducibility of complete bipartite graphs. \textit{Journal of Graph Theory}, \textbf{18}(6), 629--645, 1994.
		
		\bibitem{erd1} P. Erd\H os. Some recent results on extremal problems in graph theory, \textit{Theory
of Graphs} (Internl. Symp. Rome), 118--123, 1966.

\bibitem{erd2} P. Erd\H os. On some new inequalities concerning extremal properties of graphs, in Theory of Graphs  (ed P.
Erd\H os, G. Katona), Academic Press, New York, 77--81, 1968.

\bibitem{efr} P. Erd\H os, P. Frankl and V. R\"odl, The asymptotic number of graphs not containing a
fixed subgraph and a problem for hypergraphs having no exponent. \textit{Graphs Combin.}, \textbf{2},
113--121, 1986.

\bibitem{ge} D. Gerbner. Counting multiple graphs in generalized Tur\'an problems. \textit{arXiv preprint} arXiv:2007.11645, 2020.

\bibitem{gerbner2} D. Gerbner. On Tur\'an-good graphs. \textit{arXiv preprint} arXiv:2012.12646, 2020.

\bibitem{ger} D. Gerbner. A non-aligning variant of generalized Tur\'an problems. \textit{arXiv preprint}
arXiv:2109.02181, 2021.


		
\bibitem{gerb} D. Gerbner. Generalized Turán problems for small graphs, \textit{Discussiones Mathematicae Graph Theory}, 2021.



\bibitem{gp2} D. Gerbner, C. Palmer. Counting copies of a fixed subgraph in $ F $-free graphs. {\it European Journal of Combinatorics} {\bf 82} (2019) Article 103001. 

        \bibitem{gp3} D. Gerbner, C. Palmer. Some exact results for generalized Tur\'an problems.  \textit{European Journal of Combinatorics}, \textbf{103}, 103519, 2022.
        
        \bibitem{gypl} E. Gy\H ori, J. Pach, and M. Simonovits. On the maximal number of certain subgraphs
in $K_r$-free graphs. \textit{Graphs and Combinatorics}, \textbf{7}(1), 31--37, 1991.


\bibitem{hhl} Doudou Hei, Xinmin Hou, Boyuan Liu, Some exact results of the generalized Tur\'an numbers for paths,  \textit{arXiv preprint} arXiv:2112.14895, 2021.


\bibitem{lm} B. Lidický, K. Murphy. Maximizing five-cycles in $K_r$-free graphs. European Journal of Combinatorics, 97, 103367, 2021.

\bibitem{lpss} H. Liu, O. Pikhurko, M. Sharifzadeh, and K. Staden. Stability from graph symmetrisation
arguments with applications to inducibility.  \textit{arXiv preprint} arXiv:2012.10731,
2020.


\bibitem{mq} J. Ma, Y. Qiu, Some sharp results on the generalized Tur\'an numbers. \textit{European Journal of Combinatorics}, \textbf{84}, 103026, 2018.

\bibitem{sim} M. Simonovits. A method for solving extremal problems in graph theory, stability
problems. \textit{Theory of Graphs, Proc. Colloq., Tihany, 1966, Academic Press, New
York}, 279--319, 1968.

\bibitem{simi} M. Simonovits. Extremal graph problems with symmetrical extremal graphs. Additional chromatic conditions, \textit{Discrete Math.} \textbf{7}, 349--376, 1974.

\bibitem{T}
P. Tur\'an. Egy gr\'afelm\'eleti sz\'els\H o\'ert\'ekfeladatr\'ol. \textit{Mat. Fiz. Lapok}, \textbf{48}, 436--452, 1941.
		
\bibitem{zykov} A. A. Zykov. On some properties of linear complexes.\textit{ Matematicheskii Sbornik},\textbf{66}(2), 163--188, 1949.

\end{thebibliography}
\end{document}